\documentclass[onefignum,onetabnum]{siamart190516}

\usepackage{amsfonts,bm}
\usepackage{graphicx}
\usepackage{enumitem}
\usepackage{diagbox}
\ifpdf
\DeclareGraphicsExtensions{.eps,.pdf,.png,.jpg}
\else
\DeclareGraphicsExtensions{.eps}
\fi

\newsiamremark{remark}{Remark}
\newsiamremark{hypothesis}{Hypothesis}
\crefname{hypothesis}{Hypothesis}{Hypotheses}
\newsiamthm{claim}{Claim}
\newsiamthm{assumption}{Assumption}

\headers{Quasilinearization Methods for Nonlocal Systems}{Q. Lei and C. S. Pun}

\title{Quasilinearization Methods for Nonlocal Fully-Nonlinear Parabolic Systems}

\author{Qian Lei\thanks{School of Physical and Mathematical Sciences, Nanyang Technological University, Singapore 
		(\email{leiq0002@e.ntu.edu.sg}).}
	\and Chi Seng Pun\thanks{Corresponding author. School of Physical and Mathematical Sciences, Nanyang Technological University, Singapore 
		(\email{cspun@ntu.edu.sg}, \url{https://personal.ntu.edu.sg/cspun/}).}
	}

\usepackage{amsopn}

\ifpdf
\hypersetup{
  pdftitle={Quasilinearization Methods for Nonlocal Fully-Nonlinear Parabolic Systems},
  pdfauthor={Qian Lei and Chi Seng Pun}
}
\fi

\begin{document}

\maketitle

\begin{abstract}
  In this paper, we propose quasilinearization methods that convert nonlocal fully-nonlinear parabolic systems into the nonlocal quasilinear parabolic systems. The nonlocal parabolic systems serve as important mathematical tools for modelling the subgame perfect equilibrium solutions to time-inconsistent dynamic choice problems, which are motivated by the study of behavioral economics. Different types of nonlocal parabolic systems were studied but left behind the fully-nonlinear case and the connections among them. This paper shows the equivalence in solvability between nonlocal fully-nonlinear and the associated quasilinear systems, given their solutions are regular enough. Moreover, we establish the well-posedness results for the nonlocal quasilinear parabolic systems, so do that for the nonlocal fully-nonlinear parabolic systems. The quasilinear case alone is interesting in its own right from mathematical and modelling perspectives.
\end{abstract}



\section{Introduction} \label{sec:intro}
Conventional (or local) parabolic partial differential equations (PDEs) or systems of PDEs are closely connected with theories of stochastic controls and stochastic differential equations (SDEs). Specifically, by letting the mesh size of time interval tends to zero, the recursive equations, characterized by the dynamic programming (DP) technique, normally converge to a fully-nonlinear second-order PDE (also known as Hamilton--Jacobi--Bellman (HJB) equation), which identifies the optimal stochastic controls and the corresponding value function. In the context of SDEs, the well-known Feynman--Kac formula reveals that the adapted solutions of decoupled/coupled system of forward and backward SDEs (FBSDEs) have PDE representations. However, the link between stochastic controls and PDEs/SDEs conventionally requires time consistency, in line with Bellman's principle of optimality that implies local optima to be global optimal. We refer the readers to \cite{Yong1999} for greater details in how these three field interact.

As dynamic choice problems are being further studied, researchers find substantial empirical evidences on time inconsistency (TIC) in decision making; see \cite{Thaler1981}. When a decision-maker's preferences vary over time, being reference- or state-dependent, Bellman's principle of optimality is violated. Some well-known examples are the hyperbolic discounting \cite{Laibson1997,Frederick2002} and prospect theory \cite{Kahneman1979} in behavioural economics. For the time-inconsistent (TIC) stochastic control problems, rigorous mathematical treatment was only available about one decade ago; see \cite{He2021} for a succinct account of the latest progress. The prevailing approach, initiated in \cite{Strotz1955} and developed in \cite{Bjoerk2014,Bjoerk2017}, finds the subgame perfect equilibrium of the TIC dynamic choice problems. Fortunately, we can still obtain time-sensitive recursive equations, characterized by an extended DP instead, connecting the subproblems in time. Consequently, the TIC stochastic control problems can be characterized by a flow of fully-nonlinear second-order parabolic PDEs and the flow is connected via the equilibrium controls, resulting in nonlocality. However, such nonlocal fully-nonlinear parabolic PDEs or systems are underexplored mathematically, compared to the local ones. The aim of this paper is to propose quasilinearization methods, by which we study the well-posedness of a class of $2r$-order nonlocal parabolic PDE system for the unknown $\bm{u}:\Delta[0,T]\times\mathbb{R}^d\mapsto\mathbb{R}^m$, which takes the form
\begin{equation*}
	\left\{
	\begin{array}{rcl}
		\bm{u}_s(t,s,y) & = & \bm{F}\big(t,s,y,\left(\partial_I\bm{u}\right)_{|I|\leq 2r}(t,s,y),  \left(\partial_I\bm{u}\right)_{|I|\leq 2r}(s,s,y)\big), \\
		\bm{u}(t,0,y) & = & \bm{g}(t,y),\hfill 0\leq s\leq t\leq T,\quad y\in\mathbb{R}^d,
	\end{array}
	\right. 
\end{equation*} 
where $\Delta[0,T]=\{(t,s):0\leq s\leq t\leq T\}$, $I=(i_1,\ldots,i_j)$ is a multi-index with $j=|I|$, $\partial_I\bm{u}:=\frac{\partial^{|I|}\bm{u}}{\partial y_{i_1}\cdots\partial y_{i_d}}$. The systems above characterize a series of problems indexed by $t$ while they are connected via its dependence on $\left(\partial_I\bm{u}\right)_{|I|\leq 2r}(s,s,y)$. This kind of PDE problems have drawn attention of many researchers but they are limited to some special types of nonlocal second-order systems; see \cite{Yong2012,Wei2017,Hernandez2020,Mei2020} and the next section for the technical explanations on their specialties via a comprehensive introduction of different types of nonlocal parabolic systems. To the best of our knowledge, only \cite{Lei2021,Lei2021a} are concerned about the nonlocal fully-nonlinear parabolic PDEs and systems, investigated with a linearization approach.

We refer the readers to \cite{Lei2021,Lei2021a} for the connections between the nonlocal fully-nonlinear parabolic PDEs (resps. systems) and TIC stochastic control problems (resp. TIC stochastic differential games). In addition to the inspirations from TIC dynamic choice problems, the study on nonlocal PDEs/systems provides a new insight into solving a flow of FBSDEs (also called backward stochastic Volterra integral equation (BSVIE)) via a smooth solution $\bm{u}(\cdot,\cdot,\cdot)$ of nonlocal systems; see \cite{Wang2019,Wang2020,Hamaguchi2020,Lei2021a}. As for the Feynman--Kac formulas for a flow of FBSDEs, there has been lack of unified and general treatment for the PDE-side representations. The limited knowledge of nonlocal PDEs/systems has restricted us from solving the flow of FBSDEs by combining the solution of nonlocal systems and the adapted solution of forward SDEs in a similar manner of the classical four-step numerical scheme in \cite{Ma1994}. Hence, the study of nonlocal fully-nonlinear parabolic systems, especially on its well-posedness, is relevant and significant.  

Fully nonlinearity in PDE systems is difficult to deal with. In the classical PDE theory, \cite{Eidelman1969} introduces a so-called ``quasilinearization method" to degenerate (local) fully-nonlinear systems into more manageable quasilinear systems. On the other hand, \cite{Khudyaev1963,Sopolov1970} also utilize a variant of the quasilinearization to handle the fully nonlinearity of local systems. Although there are existing works on how to degenerate fully-nonlinear systems, it is still difficult to extend their methods from local to nonlocal cases, mainly because the degenerated nonlocal quasilinear systems for $\big(\bm{u},(\partial_I\bm{u})_{|I|=1}\big)$ or for $\big(\bm{u}_t,\bm{u}_s\big)$ contain possibly the diagonal terms $\left(\partial_I\bm{u}\right)_{|I|= 2r}(s,s,y)$. Thanks to the recent advances in \cite{Lei2021a}, where the well-posedness and Schauder's estimates of solutions $\bm{u}$ of nonlocal linear systems are established, we are able to propose new quasilinearization methods for nonlocal parabolic systems.

The major contribution of this paper is threefold. First, we extend the quasilinearization methods of \cite{Eidelman1969} from local to nonlocal systems. We also prove the equivalence in solvability between the original nonlocal fully-nonlinear systems and the degenerated nonlocal quasilinear systems. Second, we show the well-posedness of a general class of nonlocal quasilinear parabolic systems with fixed-point arguments. The quasilinear case is interesting in its own right, especially when the nonlocal parabolic systems are used to model TIC problems where the diffusion is free of controls. Together with the first point, we also obtain the well-posedness of nonlocal fully-nonlinear parabolic systems, rendering an alternative proof to \cite{Lei2021,Lei2021a}. Third, a variant of quasilinearization, in line with that of \cite{Khudyaev1963,Sopolov1970}, is also extended from local to nonlocal systems, and we illustrate how it can be used to show the well-posedness results. Two quasilinearization methods presented in this paper are different in which derivatives of the unknowns (with respect to spatial or temporal variables) are supplemented to the systems. As a result, the regularity assumptions and the regularity of the solution are different. They are also the key difference between the quasilinearization methods in this paper and the linearization method in \cite{Lei2021,Lei2021a}, i.e. when we have more restrictive regularity assumptions on system's specifications, we achieve improved regularity of the solution.

The remainder of this paper is organized as follows. Section \ref{sec:preliminaries} introduces appropriate norms and Banach spaces for the nonlocality of our interest and defines different types of nonlocal parabolic systems. The preliminary results of the well-posedness and a Schauder-type priori estimate of the solutions to nonlocal linear systems are presented. Section \ref{sec:quasilinearization} is devoted to study the quasilinearization method in the nonlocal setting. Two main results are presented: 1) the equivalence between nonlocal fully-nonlinear systems and the degenerated quasilinear systems; 2) the existence and uniqueness of nonlocal quasilinear systems by fixed-point arguments. Furthermore, a variant of the quasilinearization method is proposed and studied. Finally, Section \ref{sec:conclusion} concludes.


\section{Preliminaries for the Nonlocality}  \label{sec:preliminaries} 
In this section, we define suitable norms and Banach spaces for nonlocal problems as well as the definitions of nonlocal linear, semilinear, quasilinear, and fully-nonlinear parabolic systems. Moreover, we present some preliminary results about the nonlocal linear systems, which will lay a solid foundation for the analysis of nonlocal nonlinear systems.


\subsection{Norms and Banach Spaces for the Nonlocality} \label{sec:normspace}
We first review some definitions of spaces of H\"{o}lder continuous functions in \cite{Ladyzhanskaya1968}. Then we will adapt them to suit our nonlocal parabolic systems. 

Given $0\leq a\leq b\leq T$, we denote by $C([a,b]\times\mathbb{R}^d;\mathbb{R})$ the set of all the continuous and bounded real functions in $[a,b]\times\mathbb{R}^d$ endowed with the supremum norm $|\cdot|^{\infty}_{[a,b]\times\mathbb{R}^d}$. Wherever no confusion arises, we write $|\cdot|^\infty$ instead of $|\cdot|^{\infty}_{[a,b]\times\mathbb{R}^d}$. Next, we denote by  $C^{\frac{l}{2r},l}({[a,b]\times\mathbb{R}^d})$ the Banach space of functions $\varphi(s,y)$ that are continuous in $[a,b]\times\mathbb{R}^d$, their derivatives of the form $D^i_sD^j_y\varphi$ for $2ri+j<l$ exist, and have a finite norm defined by
\begin{eqnarray*}
	|\varphi|^{(l)}_{[a,b]\times\mathbb{R}^d} &=&\sum_{k\leq[l]}\sum_{2ri+j=k}\left|D^i_sD^j_y\varphi\right|^\infty+\sum_{2ri+j=[l]}\langle D^i_sD^j_y\varphi\rangle^{(l-[l])}_y \\
	&&+\sum_{0<l-2ri-j<2r}\langle D^i_sD^j_y\varphi\rangle^{(\frac{l-2ri-j}{2r})}_s,
\end{eqnarray*}
where $r$ is always a positive integer, $l$ a non-integer positive number with $[l]$ being its integer part, and for $0<\alpha<1$,
\begin{equation*}
	\langle\varphi\rangle^{(\alpha)}_y:=\sup\limits_{\begin{subarray}{c} a\leq s\leq b \\ y,y^\prime\in\mathbb{R}^d\end{subarray}}\frac{|\varphi(s,y)-\varphi(s,y^\prime)|}{|y-y^\prime|^\alpha}, \quad
	\langle\varphi\rangle^{(\alpha)}_s:=\sup\limits_{\begin{subarray}{c} a\leq s< s^\prime\leq b \\ y\in\mathbb{R}^d\end{subarray}}\frac{|\varphi(s,y)-\varphi(s^\prime,y)|}{|s-s^\prime|^\alpha}.
\end{equation*}
Moreover, wherever no confusion arises, we do not distinguish between $|\varphi|^{(l)}_{[a,b]\times\mathbb{R}^d}$ and $|\varphi|^{(l)}_{\mathbb{R}^d}$ for functions $\varphi(y)$ independent of $s$. 

After reviewing the conventional H\"{o}lder spaces, we modify them to fulfill requirements of studying nonlocal problems. Before that, we need to take note of the following two essential features of nonlocal (or TIC) problems: 1) the two temporal arguments $(t,s)$ has an order relation. In the context of TIC control problems and Feynman--Kac formula of flow of FBSDEs, we usually consider backward problems where $t$ and $s$ always represent the initial time and the running time, respectively. Hence, it is required that $0\leq t\leq s\leq T$. Thanks to the symmetry between forward and backward problems (see \cite[Proposition 3.1]{Lei2021a}), we consider a reverse order relation, i.e. $0\leq s\leq t\leq T$, in our study of nonlocal forward (initial-value) systems. Noting the specialty of nonlocal problems, it suffices to define functions and study the problems over a triangle region $\Delta[0,\delta]:=\left\{(t,s)\in[0,\delta]^2:~0\leq s\leq t\leq \delta\right\}$ instead of the whole rectangle region $[0,\delta]^2$; 2) for the functions $\varphi(t,s,y)$ appearing in nonlocal systems, including coefficients of differential operators, the inhomogeneous term, initial data, and potential solutions, it is desired to preserve sufficient smoothness of $\varphi$ with respect to the variables $s$ and $y$ such that most of the classical regularity results of local systems are available for nonlocal ones. As for the smoothness regarding the external time variable $t$, it is required that $\varphi(t,s,y)$ is continuously differentiable with respect to $t$ in this paper. As \cite{Lei2021a} show, the requirement is a key ingredient to promise the well-posedness of nonlocal linear systems.

Now, we are ready to define the norms and Banach spaces for nonlocal systems of unknown vector-valued functions $\bm{u}(t,s,y)=(\bm{u}^1(t,s,y),\bm{u}^2(t,s,y),\cdots,\bm{u}^m(t,s,y))^\top$. For any $t$ and $\delta$ such that $0\leq t\leq\delta\leq T$, we introduce the following norms 
\begin{eqnarray*}
	[\bm{u}]^{(l)}_{[0,\delta]} & := & \sup\limits_{t\in[0,\delta]}\sum\limits_{a\leq m}\left\{|\bm{u}^a(t,\cdot,\cdot)|^{(l)}_{[0,t]\times\mathbb{R}^d}\right\}, \\
	\lVert\bm{u}\rVert^{(l)}_{[0,\delta]} & := & \sup\limits_{t\in[0,\delta]}\sum\limits_{a\leq m}\left\{|\bm{u}^a(t,\cdot,\cdot)|^{(l)}_{[0,t]\times\mathbb{R}^d}+|\bm{u}^a_t(t,\cdot,\cdot)|^{(l)}_{[0,t]\times\mathbb{R}^d}\right\},
\end{eqnarray*}
which induce the following normed spaces, respectively,  
\begin{eqnarray*}
	\bm{\Theta}^{(l)}_{[0,\delta]} & := & \left\{\bm{u}(\cdot,\cdot,\cdot)\in C(\Delta[0,\delta]\times\mathbb{R}^d;\mathbb{R}^m):[\bm{u}]^{(l)}_{[0,\delta]}<\infty\right\}, \\
	\bm{\Omega}^{(l)}_{[0,\delta]} & := & \left\{\bm{u}(\cdot,\cdot,\cdot)\in C(\Delta[0,\delta]\times\mathbb{R}^d;\mathbb{R}^m): \Vert\bm{u}\rVert^{(l)}_{[0,\delta]}<\infty\right\}, 
\end{eqnarray*}
in which $C(\Delta[0,\delta]\times\mathbb{R^d};\mathbb{R}^m)$ is the set of all continuous and bounded $\mathbb{R}^m$-valued functions defined in $\Delta[0,\delta]\times\mathbb{R}^d$. Both $\bm{\Theta}^{(l)}_{[0,\delta]}$ and $\bm{\Omega}^{(l)}_{[0,\delta]}$ are Banach spaces under $[\bm{u}]^{(l)}_{[0,\delta]}$ and $\lVert\bm{u}\rVert^{(l)}_{[0,\delta]}$, respectively.


\subsection{Nonlocal Parabolic Systems} \label{sec:nonlocalsystems}
In this subsection, we clarify some notations and terminologies about nonlocal parabolic systems. Specifically, we categorize different types of nonlocal parabolic systems, especially the ones not fully-nonlinear. Then we restate some key mathematical results for nonlocal linear systems.  

\begin{definition}
	For a positive integer $r$ and $T\ge 0$, a $2r$-order nonlocal parabolic system (of PDEs) for the unknown $\bm{u}:\Delta[0,T]\times\mathbb{R}^d\mapsto\mathbb{R}^m$ takes the form
	\begin{equation} \label{Nonlocal fully-nonlinear systems}
		\left\{
		\begin{array}{rcl}
			\bm{u}_s(t,s,y) & = & \bm{F}\big(t,s,y,\left(\partial_I\bm{u}\right)_{|I|\leq 2r}(t,s,y),  \left(\partial_I\bm{u}\right)_{|I|\leq 2r}(s,s,y)\big), \\
			\bm{u}(t,0,y) & = & \bm{g}(t,y),\hfill 0\leq s\leq t\leq T,\quad y\in\mathbb{R}^d,  
		\end{array}
		\right. 
	\end{equation} 
	where $\Delta[0,T]=\{(t,s):0\leq s\leq t\leq T\}$, $I=(i_1,\ldots,i_j)$ be a multi-index with $j=|I|$, $\partial_I\bm{u}:=\frac{\partial^{|I|}\bm{u}}{\partial y_{i_1}\cdots\partial y_{i_d}}$, and the nonlinearity 
	\begin{equation*}
		\bm{F}:\Delta[0,T]\times\mathbb{R}^d\times\mathbb{R}^m\times\mathbb{R}^{md}\times\cdots\times\mathbb{R}^{md^{2r}}\times\mathbb{R}^m\times\mathbb{R}^{md}\times\cdots\times\mathbb{R}^{md^{2r}}\mapsto\mathbb{R}^m
	\end{equation*}
	and the initial condition $\bm{g}:[0,T]\times\mathbb{R}^d\mapsto\mathbb{R}^m$ are both given.
\end{definition}
Here, the nonlocality refers to the (diagonal) dependence of the system on the partial derivatives $\partial_I\bm{u}(s,s,y)$ (evaluated at $(s,s,y)$, i.e., diagonal line of $[0,T]^2$). When $\bm{F}$ is independent of the ``diagonal" terms $\left(\partial_I\bm{u}\right)_{|I|\leq 2r}(s,s,y)$, \eqref{Nonlocal fully-nonlinear systems} is reduced to a family of conventional (local) higher order parabolic systems parameterized by $t$.

\begin{definition}
	\noindent (1) The nonlocal parabolic system \eqref{Nonlocal fully-nonlinear systems} is called \textbf{linear} if it admits the form  
	\begin{equation} \label{Nonlocal linear systems}  
		\left\{
		\begin{array}{l}
			\bm{u}_s(t,s,y)=\sum\limits_{|I|\leq 2r}\bm{A}^{I}(t,s,y)\partial_I\bm{u}(t,s,y)+\sum\limits_{|I|\leq 2r}\bm{B}^{I}(t,s,y)\partial_I\bm{u}(s,s,y) \\
			\qquad \qquad \qquad +\bm{f}(t,s,y), \\
			\bm{u}(t,0,y)=\bm{g}(t,y),\hfill 0\leq s\leq t\leq T,\quad y\in\mathbb{R}^d. 
		\end{array}
		\right. 
	\end{equation} 
	The nonlocal linear system is homogeneous if $\bm{f}=0$. By introducing a so-called ``nonlocal differential operator" $\bm{L}=\bm{A}\bm{u}+\bm{B}\left(\bm{u}|_{t=s}\right)$, \eqref{Nonlocal linear systems} can be written as $\bm{u}_s=\bm{L}\bm{u}+\bm{f}$ with $\bm{u}|_{s=0}=\bm{g}$; 
	
	\noindent (2) The nonlocal system \eqref{Nonlocal fully-nonlinear systems} is \textbf{semilinear} if it admits the form  
	\begin{equation} \label{Nonlocal semilinear systems}  
		\left\{
		\begin{array}{l}
			\bm{u}_s(t,s,y)=\sum\limits_{|I|= 2r}\bm{A}^{I}(t,s,y)\partial_I\bm{u}(t,s,y)+\sum\limits_{|I|= 2r}\bm{B}^{I}(t,s,y)\partial_I\bm{u}(s,s,y) \\
			\qquad \qquad \qquad  
			+\bm{F}\big(t,s,y,\left(\partial_J\bm{u}\right)_{|J|\leq 2r-1}(t,s,y),\left(\partial_J\bm{u}\right)_{|J|\leq 2r-1}(s,s,y)\big), \\
			\bm{u}(t,0,y)=\bm{g}(t,y),\hfill 0\leq s\leq t\leq T,\quad y\in\mathbb{R}^d. 
		\end{array}
		\right. 
	\end{equation}
	
	\noindent (3) The nonlocal system \eqref{Nonlocal fully-nonlinear systems} is \textbf{quasilinear} if it admits the form
	\begin{equation} \label{Nonlocal quasilinear systems}  
		\left\{
		\begin{array}{l}
			\bm{u}_s(t,s,y)=\sum\limits_{|I|= 2r}\bm{A}^{I}\Big(t,s,y,\left(\partial_J\bm{u}\right)_{|J|\leq 2r-1}(t,s,y), \\
			\qquad \qquad \qquad \qquad \qquad \qquad \quad \left(\partial_J\bm{u}\right)_{|J|\leq 2r-1}(s,s,y)\Big)\partial_I\bm{u}(t,s,y) \\
			\qquad \qquad \qquad
			+\sum\limits_{|I|= 2r}\bm{B}^{I}\Big(t,s,y,\left(\partial_J\bm{u}\right)_{|J|\leq 2r-1}(t,s,y), \\
			\qquad \qquad \qquad \qquad \qquad \qquad \qquad ~~ \left(\partial_J\bm{u}\right)_{|J|\leq 2r-1}(s,s,y)\Big)\partial_I\bm{u}(s,s,y) \\
			\qquad \qquad \qquad 
			+\bm{F}\big(t,s,y,\left(\partial_J\bm{u}\right)_{|J|\leq 2r-1}(t,s,y),\left(\partial_J\bm{u}\right)_{|J|\leq 2r-1}(s,s,y)\big), \\
			\bm{u}(t,0,y)=\bm{g}(t,y),\hfill 0\leq s\leq t\leq T,\quad y\in\mathbb{R}^d. 
		\end{array}
		\right. 
	\end{equation} 
	
	\noindent (4) The nonlocal system \eqref{Nonlocal fully-nonlinear systems} is \textbf{fully-nonlinear} if its dependence on the highest order derivatives $\left(\partial_I\bm{u}\right)_{|I|=2r}(t,s,y)$ and/or $\left(\partial_I\bm{u}\right)_{|I|=2r}(s,s,y)$ is nonlinear. 
\end{definition}

It is noteworthy to distinguish the nonlocal semilinear systems \eqref{Nonlocal semilinear systems} and its special case:
\begin{equation} \label{Previous nonlocal examples}  
	\left\{
	\begin{array}{l}
		\bm{u}_s(t,s,y)=\sum\limits_{|I|= 2}\bm{A}^{I}(t,s,y)\partial_I\bm{u}(t,s,y)\\
		\qquad \qquad \qquad +\bm{F}\big(t,s,y,\left(\partial_J\bm{u}\right)_{|J|\leq 1}(t,s,y),\left(\partial_J\bm{u}\right)_{|J|\leq 1}(s,s,y)\big), \\
		\bm{u}(t,0,y)=\bm{g}(t,y),\hfill 0\leq s\leq t\leq T,\quad y\in\mathbb{R}^d,
	\end{array}
	\right. 
\end{equation}
which is studied in \cite{Yong2012,Wei2017,Hernandez2020,Mei2020,Wang2020,Hamaguchi2020}. Since there is no highest-order terms at $(s,s,y)$ in \eqref{Previous nonlocal examples}, i.e. $\left(\partial_I\bm{u}\right)_{|I|=2}(s,s,y)$, the authors of \cite{Yong2012,Wei2017,Wang2020} are able to construct a contraction via the solutions to conventional parameterized parabolic equations by replacing all lower order terms with pre-specified functions and obtain the well-posedness of \eqref{Previous nonlocal examples} with fixed-point arguments. However, the appearance of $\sum_{|I|= 2r}\bm{B}^{I}(t,s,y)\partial_I\bm{u}(s,s,y)$ in \eqref{Nonlocal semilinear systems} poses essential analytical challenges. Specifically, if a similar methodology is followed, we should consider the following mapping from $\bm{u}$ to $\bm{U}$ for nonlocal fully-nonlinear systems:
\begin{equation*} 
	\left\{
	\begin{array}{lr}
		\bm{U}_s(t,s,y)=\bm{F}\big(t,s,y,\left(\partial_I\bm{U}\right)_{|I|\leq 2r}(t,s,y),  \left(\partial_I\bm{u}\right)_{|I|\leq 2r}(s,s,y)\big), \\
		\bm{U}(t,0,y)=\bm{g}(t,y),\quad 0\leq s\leq t\leq T,\quad y\in\mathbb{R}^d.  
	\end{array}
	\right. 
\end{equation*} 
Noting that the inputs $\left(\partial_I\bm{u}\right)_{|I|\leq 2r}(s,s,y)$ are of the same order as the outputs $\left(\partial_I\bm{U}\right)_{|I|\leq 2r}(t,s,y)$, it is infeasible to prove the mapping to be a contraction. This technical difficulty motivates us to lower the order of derivatives in nonlocal fully-nonlinear systems. The quasilinearization serves this purpose in the literature.

In what follows, we will take advantage of the well-posedness results of \eqref{Nonlocal linear systems} to prove the existence and uniqueness of \eqref{Nonlocal semilinear systems} and \eqref{Nonlocal quasilinear systems}, and develop the quasilinearization methods with nonlocality to degenerate \eqref{Nonlocal fully-nonlinear systems} as \eqref{Nonlocal quasilinear systems} with the mathematical claims preserved.

\begin{definition}
	The nonlocal parabolic system \eqref{Nonlocal fully-nonlinear systems} is said to satisfy the \textbf{uniform ellipticity condition} if there exists $\lambda>0$ such that
	\begin{eqnarray}
		&& (-1)^{r-1}\sum_{a,b\leq m,|I|=2r}\partial_I \bm{F}^a_b(t,s,y,z)\xi_{i_1}\cdots\xi_{i_{2r}}v^av^b \geq \lambda|\xi|^{2r}|v|^2,  \label{Uniform ellipticity condition of local part} \\
		&&
		(-1)^{r-1}\sum_{a,b\leq m,|I|=2r}\left(\partial_I \bm{F}^a_b+\partial_I \overline{\bm{F}}^a_b\right)(t,s,y,z)\xi_{i_1}\cdots\xi_{i_{2r}}v^av^b \geq   \lambda|\xi|^{2r}|v|^2   \label{Uniform ellipticity condition of nonlocal part} 
	\end{eqnarray}
	hold uniformly with respect to $(t,s,y,z)$ for any $\xi=(\xi_1,\ldots,\xi_d)^\top\in\mathbb{R}^d$ and $v\in\mathbb{R}^m$, where $z\in \mathbb{R}^m\times\mathbb{R}^{md}\times\cdots\times\mathbb{R}^{md^{2r}}\times\mathbb{R}^m\times\mathbb{R}^{md}\times\cdots\times\mathbb{R}^{md^{2r}}$ represents the rest of arguments in $\bm{F}$ other than $(t,s,y)$ and $\partial_I\bm{F}^a_b$ (resp. $\partial_I\overline{\bm{F}}^a_b$) denotes the first order derivative of the $a$-th entry $\bm{F}^a$ of $\bm{F}$ with respect to the variable $\partial_I\bm{u}^b(t,s,y)$ (resp.  $\partial_I\bm{u}^b(s,s,y)$). 
\end{definition}

Compared with the classical uniform ellipticity condition, there is an additional inequality \eqref{Uniform ellipticity condition of nonlocal part} for nonlocal highest order term $\left(\partial_I\bm{u}\right)_{|I|=2r}(s,s,y)$. In some earlier works \cite{Yong2012,Wei2017,Wang2020}, since their $\partial_I \overline{\bm{F}}^a_b$ always equals to zero, the conditions \eqref{Uniform ellipticity condition of local part} and \eqref{Uniform ellipticity condition of nonlocal part} coincide. Given the two non-degenerated conditions, a variety of regularity results of PDEs still hold true in the nonlocal framework.  

Under some regularity assumptions on the coefficients, the inhomogeneous term $\bm{f}$, and the initial data $\bm{g}$, the nonlocal linear PDE \eqref{Nonlocal linear systems} is well-posed and the behaviours of its solutions are controlled by a Schauder-type estimate. These claims, proved in \cite{Lei2021a} and recapped below, establishes a foundation for the further analysis of nonlocal fully-nonlinear systems in the next section. 

\begin{theorem} \label{Schauder estimates}
	Consider a nonlocal linear system $\bm{u}_s=\bm{L}\bm{u}+\bm{f}$ with $\bm{u}|_{s=0}=\bm{g}$ given in \eqref{Nonlocal linear systems}. Suppose that the coefficients $\bm{A}^I$ and $\bm{B}^I$ of $\bm{L}$ satisfy the uniform ellipticity condition \eqref{Uniform ellipticity condition of local part}-\eqref{Uniform ellipticity condition of nonlocal part} and belong to $\bm{\Omega}^{{(k+\alpha)}}_{[0,T]}$, $\bm{f}\in\bm{\Omega}^{{(k+\alpha)}}_{[0,T]}$, and $\bm{g}\in\bm{\Omega}^{{(2r+k+\alpha)}}_{[0,T]}$ for an integer $k\leq 2r-1$. Then there exist $\delta>0$ and a unique $\bm{u}\in\bm{\Omega}^{{(2r+k+\alpha)}}_{[0,\delta]}$ satisfying \eqref{Nonlocal linear systems} in $\Delta[0,\delta]\times\mathbb{R}^d$. Furthermore, the following Schauder-type estimate holds 
	\begin{equation} \label{Estimates of solutions of nonlocal linear PDEs} 
		\lVert \bm{u}\rVert^{(2r+k+\alpha)}_{[0,\delta]}\leq C\left(\lVert \bm{f}\rVert^{(k+\alpha)}_{[0,\delta]}+\lVert \bm{g}\rVert^{(2r+k+\alpha)}_{[0,\delta]}\right),  
	\end{equation}
	where $C$ is a constant depending only on coefficients of the nonlocal differential operator $\bm{L}$. 
\end{theorem} 

This well-posedness result of nonlocal linear systems are similar to the classical ones in \cite{Ladyzhanskaya1968,Lunardi1995,Solonnikov1965}. It shows not only that the Banach space $\bm{\Omega}^{(l)}_{[0,\delta]}$ is suitable for the problem of our interest in the sense that the mapping of $\bm{L}$ is closed among these defined spaces, but also that the regularities of solutions can be improved accordingly in line with that of $\bm{A}^I$, $\bm{B}^I$, $\bm{f}$, and $\bm{g}$. 


\section{Quasilinearization Methods for Nonlocal Parabolic Systems} \label{sec:quasilinearization} 
This section first present a quasilinearization method that degenerates nonlocal fully-nonlinear systems into quasilinear ones. Then, we adopt fixed-point arguments to prove the existence and uniqueness of nonlocal quasilinear systems \eqref{Nonlocal quasilinear systems}. Finally, inspired by the degeneracy approach, a variant of the quasilinearization method is investigated to study nonlocal fully-nonlinear systems from a different perspective.   


\subsection{Quasilinear Counterparts of Nonlocal Fully-Nonlinear Parabolic Systems} \label{sec:degeneration}
The major difficulty of showing the well-posedness of \eqref{Nonlocal fully-nonlinear systems} is to deal with the highest order terms $\left(\partial_I\bm{u}\right)_{|I|=2r}(t,s,y)$ and $\left(\partial_I\bm{u}\right)_{|I|=2r}(s,s,y)$. Without loss of generality, we consider a simplied nonlinearity $\bm{F}$ of \eqref{Nonlocal fully-nonlinear systems} without the lower order terms $\left(\partial_I\bm{u}\right)_{|I|\leq 2r-1}(t,s,y)$  and $\left(\partial_I\bm{u}\right)_{|I|\leq 2r-1}(s,s,y)$ and let $r=m=1$ such that \eqref{Nonlocal fully-nonlinear systems} is reduced to a nonlocal fully-nonlinear second-order PDE for an unknown real-valued function $u$:
\begin{equation} \label{Simplied nonlocal fully-nonlinear equation}   
	\left\{
	\begin{array}{rcl}
		u_s(t,s,y)&=&F\big(t,s,y,u_{yy}(t,s,y),u_{yy}(s,s,y)\big), ~~~~~~ \\
		u(t,0,y)&=&g(t,y),\hfill 0\leq s\leq t\leq T,\quad y\in\mathbb{R}^d.
	\end{array}
	\right.
\end{equation} 
A similar analysis below can be easily extended to a more general $\bm{F}$ and higher order systems. 

Suppose that the solution $u$ to \eqref{Simplied nonlocal fully-nonlinear equation} is thrice continuously differentiable with respect to $y$. By differentiating both sides of \eqref{Simplied nonlocal fully-nonlinear equation} with respect to $y_k$ for $k=1,\ldots,d$ and denoting $v^{(k)}=\frac{\partial u}{\partial y_k}$, we can obtain a system of PDEs for $\left(u,v^{(1)},v^{(2)},\cdots,v^{(d)}\right)$: 
\begin{equation*}  
	\left\{
	\begin{array}{l}
		u_s(t,s,y)=F\big(t,s,y,u_{yy}(t,s,y),u_{yy}(s,s,y)\big), \\
		v^{(k)}_s(t,s,y)=F_{y_k}\left(u\right)+\sum^d\limits_{i,j=1}F_{q_{ij}}(u)v^{(k)}_{y_iy_j}(t,s,y)+\sum^d\limits_{i,j=1}F_{n_{ij}}(u)v^{(k)}_{y_iy_j}(s,s,y),\\
		\hfill k=1,\ldots,d, \\
		\left(u,v^{(1)},\cdots,v^{(d)}\right)(t,0,y)=(g,g_{y_1},\cdots,g_{y_d})(t,y),\quad 0\leq s\leq t\leq T,\quad y\in\mathbb{R}^d,
	\end{array}
	\right. 
\end{equation*} 
where $F_{y_k}(u)$, $F_{q_{ij}}(u)$, and $F_{n_{ij}}(u)$ represent the first order derivative of $F$ with respect to $y_k$, $u_{y_iy_j}(t,s,y)$, and $u_{y_iy_j}(s,s,y)$, respectively, while they are all evaluated at the point \\ $(t,s,y,u_{yy}(t,s,y),u_{yy}(s,s,y))$. Moreover, it can be rewritten as  
\begin{equation} \label{Induced system of v} 
	\left\{
	\begin{array}{l}
		u_s(t,s,y)=\sum^d\limits_{i,j=1}u_{y_iy_j}(t,s,y)+\sum^d\limits_{i,j=1}u_{y_iy_j}(s,s,y) \\
		\qquad\qquad\quad 
		+F\big(t,s,y,v_{y}(t,s,y),v_{y}(s,s,y)\big)-\sum^d\limits_{i,j=1}v^{(i)}_{y_j}(t,s,y)-\sum^d\limits_{i,j=1}v^{(i)}_{y_j}(s,s,y), \\
		v^{(k)}_s(t,s,y)=F_{y_k}\left(v\right)+\sum^d\limits_{i,j=1}F_{q_{ij}}(v)v^{(k)}_{y_iy_j}(t,s,y)+\sum^d\limits_{i,j=1}F_{n_{ij}}(v)v^{(k)}_{y_iy_j}(s,s,y), \\
		\hfill k=1,\ldots,d, \\
		\left(u,v^{(1)},\cdots,v^{(d)}\right)(t,0,y)=(g,g_{y_1},\cdots,g_{y_d})(t,y),\hfill 0\leq s\leq t\leq T,\quad y\in\mathbb{R}^d.
	\end{array}
	\right. 
\end{equation} 
where $F_{y_k}(v)$, $F_{q_{ij}}(v)$ and $F_{n_{ij}}(v)$ are all evaluated at $(t,s,y,v_y(t,s,y),v_y(s,s,y))$. The induced system \eqref{Induced system of v} admits the form of a nonlocal quasilinear system \eqref{Nonlocal quasilinear systems}.

Next, the following lemma shows the equivalence between \eqref{Simplied nonlocal fully-nonlinear equation} and \eqref{Induced system of v}.  
\begin{lemma} \label{lem:equivalence}
	Suppose that in solving Problems \eqref{Simplied nonlocal fully-nonlinear equation} and \eqref{Induced system of v}, we look for sufficiently regular solutions $u$ and $(u,v^{(1)},v^{(2)},\cdots,v^{(d)})$, respectively, such that $u$ is thrice continuously differentiable with respect to $y$. Then Problems \eqref{Simplied nonlocal fully-nonlinear equation} and \eqref{Induced system of v} are equivalent (locally), i.e.
	\begin{enumerate} 
		\item if $u$ (locally) solves \eqref{Simplied nonlocal fully-nonlinear equation}, then $(u,\frac{\partial u}{\partial y_1},\frac{\partial u}{\partial y_2},\cdots,\frac{\partial u}{\partial y_d})$ (locally) solves \eqref{Induced system of v}; 
		\item if $(u,v^{(1)},v^{(2)},\cdots,v^{(d)})$ (locally) solves \eqref{Induced system of v}, then $u$ (locally) solves \eqref{Simplied nonlocal fully-nonlinear equation} and $v^{(k)}=\frac{\partial u}{\partial y_k}$. 
	\end{enumerate}
\end{lemma}
\begin{proof}
	The first claim is obvious given that $u$ is thrice continuously differentiable in $y$. We only prove the second claim.  
	To this end, we first show that
	\begin{equation} \label{Exchange condition}
		v^{(k)}_{y_l}=v^{(l)}_{y_k}.
	\end{equation}
	for $k,l=1,\ldots,d$ while $(u,v^{(1)},v^{(2)},\cdots,v^{(d)})$ solves \eqref{Induced system of v}. Differentiating the equation of $v^{(k)}$ of \eqref{Induced system of v} with $y_l$ and differentiating $v^{(l)}$ with ${y_k}$ yield  
	\begin{equation*} 
		\left\{
		\begin{array}{l}
			\big(v^{(k)}_{y_l}\big)_s(t,s,y)=F_{y_ky_l}(v) \\
			\qquad \qquad \qquad ~ +\sum^d\limits_{i,j=1}F_{y_kq_{ij}}(v)v^{(i)}_{y_jy_l}(t,s,y)+\sum^d\limits_{i,j=1}F_{y_kn_{ij}}(v)v^{(i)}_{y_jy_l}(s,s,y) \\
			\qquad \qquad \qquad ~ +\sum^d\limits_{i,j=1}v^{(k)}_{y_iy_j}(t,s,y) \bigg[F_{y_lq_{ij}}(v)+F_{q_{ij}q_{ij}}(v)v^{(i)}_{y_jy_l}(t,s,y) \\
			\qquad \qquad \qquad \qquad \qquad \qquad \qquad \qquad \qquad \qquad \qquad +F_{q_{ij}n_{ij}}(v)v^{(i)}_{y_jy_l}(s,s,y)\bigg]   \\
			\qquad \qquad \qquad ~ +\sum^d\limits_{i,j=1}v^{(k)}_{y_iy_j}(s,s,y) \bigg[F_{y_ln_{ij}}(v)+F_{n_{ij}q_{ij}}(v)v^{(i)}_{y_jy_l}(t,s,y) \\
			\qquad \qquad \qquad \qquad \qquad \qquad \qquad \qquad \qquad \qquad \qquad +F_{n_{ij}n_{ij}}(v)v^{(i)}_{y_jy_l}(s,s,y)\bigg]   \\
			\qquad \qquad \qquad ~ +\sum^d\limits_{i,j=1}F_{q_{ij}}(v)\left(v^{(k)}_{y_l}\right)_{y_iy_j}(t,s,y)+\sum^d\limits_{i,j=1}F_{n_{ij}}(v)\left(v^{(k)}_{y_l}\right)_{y_iy_j}(s,s,y), \\
			\big(v^{(l)}_{y_k}\big)_s(t,s,y)=F_{y_ly_k}(v) \\
			\qquad \qquad \qquad ~ +\sum^d\limits_{i,j=1}F_{y_lq_{ij}}(v)v^{(i)}_{y_jy_k}(t,s,y)+\sum^d\limits_{i,j=1}F_{y_ln_{ij}}(v)v^{(i)}_{y_jy_k}(s,s,y) \\
			\qquad \qquad \qquad ~ +\sum^d\limits_{i,j=1}v^{(l)}_{y_iy_j}(t,s,y)\bigg[F_{y_kq_{ij}}(v)+F_{q_{ij}q_{ij}}(v)v^{(i)}_{y_jy_k}(t,s,y) \\
			\qquad \qquad \qquad \qquad \qquad \qquad \qquad \qquad \qquad \qquad \qquad +F_{q_{ij}n_{ij}}(v)v^{(i)}_{y_jy_k}(s,s,y)\bigg]    \\
			\qquad \qquad \qquad ~ +\sum^d\limits_{i,j=1}v^{(l)}_{y_iy_j}(s,s,y)\bigg[F_{y_kn_{ij}}(v)+F_{n_{ij}q_{ij}}(v)v^{(i)}_{y_jy_k}(t,s,y) \\
			\qquad \qquad \qquad \qquad \qquad \qquad \qquad \qquad \qquad \qquad \qquad +F_{n_{ij}n_{ij}}(v)v^{(i)}_{y_jy_k}(s,s,y)\bigg]    \\
			\qquad \qquad \qquad ~ +\sum^d\limits_{i,j=1}F_{q_{ij}}(v)\left(v^{(l)}_{y_k}\right)_{y_iy_j}(t,s,y)+\sum^d\limits_{i,j=1}F_{n_{ij}}(v)\left(v^{(l)}_{y_k}\right)_{y_iy_j}(s,s,y), \\
			(v^{(k)}_{y_l},v^{(l)}_{y_k})(t,0,y)=(g_{y_ky_l},g_{y_ly_k})(t,y),\hfill 0\leq s\leq t\leq T,~y\in\mathbb{R}^d,~k,l=1,\ldots,d. 
		\end{array}
		\right. 
	\end{equation*} 
	Then, by noting the symmetry between $F_{q_{ij}n_{ij}}$ and $F_{n_{ij}q_{ij}}$, i.e. $F_{q_{ij}n_{ij}}(v)=F_{n_{ij}q_{ij}}(v)$, the difference between $v^{(k)}_{y_l}$ and $v^{(l)}_{y_k}$ satisfies the nonlocal system: 
	\begin{equation} \label{eq:diff2ndv}
		\left\{
		\begin{array}{l}
			\left(v^{(k)}_{y_l}-v^{(l)}_{y_k}\right)_s(t,s,y)=\sum^d\limits_{i,j=1}F_{y_kq_{ij}}(v)\left(v^{(i)}_{y_l}-v^{(l)}_{y_i}\right)_{y_j}(t,s,y) \\
			
			\qquad\qquad\qquad\qquad\qquad\quad +\sum^d\limits_{i,j=1}F_{y_kn_{ij}}(v)\left(v^{(i)}_{y_l}-v^{(l)}_{y_i}\right)_{y_j}(s,s,y) \\
			
			\qquad\qquad\qquad\qquad\qquad\quad +\sum^d\limits_{i,j=1}F_{y_lq_{ij}}(v)\left(v^{(k)}_{y_i}-v^{(i)}_{y_k}\right)_{y_j}(t,s,y) \\
			
			\qquad\qquad\qquad\qquad\qquad\quad +\sum^d\limits_{i,j=1}F_{y_ln_{ij}}(v)\left(v^{(k)}_{y_i}-v^{(i)}_{y_k}\right)_{y_j}(s,s,y) \\
			
			\qquad\qquad\qquad\qquad\qquad\quad +\sum^d\limits_{i,j=1}F_{q_{ij}q_{ij}}(v)v^{(i)}_{y_jy_l}(t,s,y)\left(v^{(k)}_{y_i}-v^{(i)}_{y_k}\right)_{y_j}(t,s,y)  \\
			
			\qquad\qquad\qquad\qquad\qquad\quad +\sum^d\limits_{i,j=1}F_{q_{ij}q_{ij}}(v)v^{(i)}_{y_jy_k}(t,s,y)\left(v^{(i)}_{y_l}-v^{(l)}_{y_i}\right)_{y_j}(t,s,y) \\
			
			\qquad\qquad\qquad\qquad\qquad\quad +\sum^d\limits_{i,j=1}F_{n_{ij}n_{ij}}(v)v^{(i)}_{y_jy_l}(s,s,y)\left(v^{(k)}_{y_i}-v^{(i)}_{y_k}\right)_{y_j}(s,s,y)  \\
			
			\qquad\qquad\qquad\qquad\qquad\quad +\sum^d\limits_{i,j=1}F_{n_{ij}n_{ij}}(v)v^{(i)}_{y_jy_k}(s,s,y)\left(v^{(i)}_{y_l}-v^{(l)}_{y_i}\right)_{y_j}(s,s,y) \\
			
			\qquad\qquad\qquad\qquad\qquad\quad +\sum^d\limits_{i,j=1}F_{q_{ij}n_{ij}}(v)v^{(i)}_{y_jy_l}(s,s,y)\left(v^{(k)}_{y_i}-v^{(i)}_{y_k}\right)_{y_j}(t,s,y)  \\
			
			\qquad\qquad\qquad\qquad\qquad\quad +\sum^d\limits_{i,j=1}F_{q_{ij}n_{ij}}(v)v^{(i)}_{y_jy_k}(t,s,y)\left(v^{(i)}_{y_l}-v^{(l)}_{y_i}\right)_{y_j}(s,s,y) \\
			
			\qquad\qquad\qquad\qquad\qquad\quad +\sum^d\limits_{i,j=1}F_{n_{ij}q_{ij}}(v)v^{(i)}_{y_jy_l}(t,s,y)\left(v^{(k)}_{y_i}-v^{(i)}_{y_k}\right)_{y_j}(s,s,y)  \\
			
			\qquad\qquad\qquad\qquad\qquad\quad +\sum^d\limits_{i,j=1}F_{n_{ij}q_{ij}}(v)v^{(i)}_{y_jy_k}(s,s,y)\left(v^{(i)}_{y_l}-v^{(l)}_{y_i}\right)_{y_j}(t,s,y) \\
			
			\qquad\qquad\qquad\qquad\qquad\quad +\sum^d\limits_{i,j=1}F_{q_{ij}}(v)\left(v^{(k)}_{y_l}-v^{(l)}_{y_k}\right)_{y_iy_j}(t,s,y) \\
			
			\qquad\qquad\qquad\qquad\qquad\quad +\sum^d\limits_{i,j=1}F_{n_{ij}}(v)\left(v^{(k)}_{y_l}-v^{(l)}_{y_k}\right)_{y_iy_j}(s,s,y), \\
			
			\left(v^{(k)}_{y_l}-v^{(l)}_{y_k}\right)(t,0,y)=0,\hfill 0\leq s\leq t\leq T,\quad y\in\mathbb{R}^d, \quad k,l=1,\ldots,d,
		\end{array}
		\right. 
	\end{equation} 
	which is a nonlocal linear system for $v^{(k)}_{y_l}-v^{(l)}_{y_k}$ for any $k,l=1,\ldots,d$. Suppose that both $F$ and $g$ are regular enough and that $(u,v^{(1)},\cdots,v^{(d)})$ is a solution of \eqref{Induced system of v}, we have a trivial solution to \eqref{eq:diff2ndv} and $v^{(k)}_{y_l}=v^{(l)}_{y_k}$ by the uniqueness of the solution to nonlocal linear systems in Theorem \ref{Schauder estimates}. Hence, the condition \eqref{Exchange condition} holds.
	
	We are now ready to show the second claim. According to \eqref{Induced system of v}, we have 
	\begin{equation*} 
		\left\{
		\begin{array}{l}
			\left(\frac{\partial u}{\partial y_k}\right)_s(t,s,y)=\sum^d\limits_{i,j=1}\left(\frac{\partial u}{\partial y_k}\right)_{y_iy_j}(t,s,y)+\sum^d\limits_{i,j=1}\left(\frac{\partial u}{\partial y_k}\right)_{y_iy_j}(s,s,y) \\
			\qquad \qquad \qquad \qquad +F_{y_k}(v)+\sum^d\limits_{i,j=1}F_{q_{ij}}(v)v^{(i)}_{y_jy_k}(t,s,y)+\sum^d\limits_{i,j=1}F_{n_{ij}}(v)v^{(i)}_{y_jy_k}(s,s,y) \\
			\qquad \qquad \qquad \qquad -\sum^d\limits_{i,j=1}v^{(i)}_{y_jy_k}(t,s,y)-\sum^d\limits_{i,j=1}v^{(i)}_{y_jy_k}(s,s,y), \\
			\left(\frac{\partial u}{\partial y_k}\right)(t,0,y)=g_{y_k}(t,y),\hfill 0\leq s\leq t\leq T,\quad y\in\mathbb{R}^d, \quad k=1,\ldots,d. 
		\end{array}
		\right. 
	\end{equation*}
	With the condition \eqref{Exchange condition}, the difference between $\frac{\partial u}{\partial y_k}$ and $v^{(k)}$ satisfies 
	\begin{equation*} 
		\left\{
		\begin{array}{l}
			\left(\frac{\partial u}{\partial y_k}-v^{(k)}\right)_s(t,s,y)=\sum^d\limits_{i,j=1}\left(\frac{\partial u}{\partial y_k}-v^{(k)}\right)_{y_iy_j}(t,s,y) \\
			\qquad \qquad \qquad \qquad \qquad \qquad \qquad +\sum^d\limits_{i,j=1}\left(\frac{\partial u}{\partial y_k}-v^{(k)}\right)_{y_iy_j}(s,s,y), \\
			\left(\frac{\partial u}{\partial y_k}-v^{(k)}\right)(t,0,y)=0, \qquad 0\leq s\leq t\leq T,\quad y\in\mathbb{R}^d, \quad k=1,\ldots,d,
		\end{array}
		\right. 
	\end{equation*} 
	which implies that $\frac{\partial u}{\partial y_k}=v^{(k)}$, again by the uniqueness of and Schauder's estimate of solutions to nonlocal linear systems in Theorem \ref{Schauder estimates}. Consequently, from the first equation of \eqref{Induced system of v}, we find that $u$ solves \eqref{Simplied nonlocal fully-nonlinear equation}. The result follows.
\end{proof}

The equivalence results in Lemma \ref{lem:equivalence} are extendable for a more general cases. Specifically, the quasilinearization method can degenerate a nonlocal fully-nonlinear system \eqref{Nonlocal fully-nonlinear systems} to a nonlocal quasilinear system \eqref{Nonlocal quasilinear systems}. Intuitively, the extension requires that the solutions of \eqref{Nonlocal fully-nonlinear systems} and \eqref{Nonlocal quasilinear systems} are $(2r+1)$-continuously differentiable with respect to $y$.

We remark here that since the proof of Lemma \ref{lem:equivalence} leveraged on the well-posedness results of nonlocal linear systems in \cite{Lei2021a}, which is obtained for maximally defined solutions over an interval $[0,\tau]$ for the largest possible $\tau$ but possibly less than $T$, the claims of Lemma \ref{lem:equivalence} will also be limited to such an interval and thus the equivalence is local in this sense. Though it is beyond the scope of this paper, we do believe that the well-posedness of nonlocal linear systems can be extended from maximally defined solution to that at large. Our subsequent analyses will also be concerned about the maximally defined solutions.


\subsection{Well-posedness of Nonlocal Quasilinear Systems}
In this subsection, we will prove the existence and uniqueness of nonlocal quasilinear systems \eqref{Nonlocal quasilinear systems} by fixed-point arguments. Given the well-posedness of nonlocal linear systems \eqref{Nonlocal linear systems}, we can establish a contraction by linearizing the quasilinear systems. Furthermore, with the regularity assumptions on coefficients, inhomogeneous term, and initial data improving, the solutions of nonlocal quasilinear systems will be equipped with higher level of regularity in the sense that they are $(2r+1)$-continuously differentiability with respect to $y$, which is exactly required by the quasilinearization method. 

To make use of the results of nonlocal linear systems in Section \ref{sec:preliminaries}, we require certain regularity assumptions on $\bm{A}^I$, $\bm{B}^I$, $\bm{F}$ and $\bm{g}$ of \eqref{Nonlocal quasilinear systems}. We denote a generic $\bm{\mathcal{H}}$, which could be $\bm{A}^I$, $\bm{B}^I$, or $\bm{F}$, and denote by $\bm{\mathcal{H}}_{\bm{\mathcal{X}}}$ and $\bm{\mathcal{H}}_{\bm{\mathcal{XY}}}$ its first- and second-order derivatives, respectively, with respect to $t$ or the variables $\partial_I\bm{u}(t,s,y)$, $\partial_I\bm{u}(s,s,y)$ indicated in Tables \ref{tab:table1} and \ref{tab:table2}. Then the vector/matrix-valued mapping $(t,s,y,z)\mapsto \bm{\mathcal{H}}(t,s,y,z)$ is defined in $\Pi=\Delta[0,T]\times\mathbb{R}^d\times B(\overline{z},R_0)$ for some positive constant $R_0$, where $\overline{z}\in\mathbb{R}^m\times\mathbb{R}^{md}\times\cdots\times\mathbb{R}^{md^{2r-1}}\times\mathbb{R}^m\times\mathbb{R}^{md}\times\cdots\times\mathbb{R}^{md^{2r-1}}$.  
\begin{table}[!ht] 
	\centering
	\begin{tabular}{c| c c c c c c c c c}
		\hline
		$\mathcal{X}$ & $t$ & $s$ & $y$ & $(\partial_I\bm{u})_{|I|\leq 2r-1}(t,s,y)$ & $(\partial_I\bm{u})_{|I|\leq 2r-1}(s,s,y)$ \\ 
		\hline 
		$\bm{\mathcal{H}}_\mathcal{X}$ & $\surd$ & & & $\surd$ & $\surd$ \\ 
		\hline 
	\end{tabular}
	\caption{First-order derivatives of $\bm{\mathcal{H}}$ required to be H\"{o}lder and Lipschitz continuous}
	\label{tab:table1}
\end{table} 

\begin{table}[!ht] 
	\centering
	\begin{tabular}{c| c c c c c c c c c}
		\hline
		\diagbox{$\mathcal{X}$}{$\bm{\mathcal{H}}_{\mathcal{X}\mathcal{Y}}$}{$\mathcal{Y}$} & $t$ & $s$ & $y$ & $(\partial_I\bm{u})_{|I|\leq 2r-1}(t,s,y)$ & $(\partial_I\bm{u})_{|I|\leq 2r-1}(s,s,y)$ \\ 
		\hline 
		$t$ &  &  &  & $\surd$ & $\surd$ \\ 
		$s$ &  &  &  &  & \\
		$y$ &  &  &  &  & \\
		$(\partial_I\bm{u})_{|I|\leq 2r-1}(t,s,y)$ & $\surd$ &  &  & $\surd$ & $\surd$ \\
		$(\partial_I\bm{u})_{|I|\leq 2r-1}(s,s,y)$ & $\surd$ &  &  & $\surd$ & \\
		\hline 
	\end{tabular}
	\caption{Second-order derivatives of $\bm{\mathcal{H}}$ required to be H\"{o}lder and Lipschitz continuous}
	\label{tab:table2}
\end{table}

\begin{assumption} \label{assumption}
	By introducing $\mathbb{H}\in\left\{\bm{\mathcal{H}},\bm{\mathcal{H}}_{\bm{\mathcal{X}}},\bm{\mathcal{H}}_{\bm{\mathcal{XY}}}\right\}$, it is required that $\mathbb{H}$ is continuous with respect to its all arguments, locally Lipschitz continuous with respect to $z$, and locally $C^{\frac{\alpha}{2},\alpha}$-H\"{o}lder continuous with respect to $(s,y)$ uniformly with respect to the other variables. Specifically, together with the uniform ellipticity condition, $\mathbb{H}$ needs to satisfy the following conditions:
	\begin{enumerate}[label=(\roman*)]
		\item \textbf{(Ellipticity condition)} for any $\xi=(\xi_1,\ldots,\xi_d)\in\mathbb{R}^d$ and $v=(v^1,\ldots.v^m)\in\mathbb{R}^m$, there exists $\lambda>0$ such that
		\begin{eqnarray*} \label{Ellipticity conditions of AB}
			&& (-1)^{r-1}\sum_{|I|=2r}\left[\bm{A}^I(t,s,y,z)v,v\right]_{\mathbb{R}^m}\xi_{i_1}\cdots\xi_{i_{2r}} > \lambda|\xi|^{2r}|v|^2, \\
			&&
			(-1)^{r-1}\sum_{|I|=2r}\left[\left(\bm{A}^I+\bm{B}^I\right)(t,s,y,z)v,v\right]_{\mathbb{R}^m}\xi_{i_1}\cdots\xi_{i_{2r}} > \lambda|\xi|^{2r}|v|^2
		\end{eqnarray*}
		hold uniformly with respect to $(t,s,y,z)\in\Pi$, where $[\cdot,\cdot]_{\mathbb{R}^m}$ denotes the standard scalar product in $\mathbb{R}^m$;
		
		\item \textbf{(H\"{o}lder continuity)} for every $\delta\geq 0$ and $z\in B(\overline{z},R_0)$, there exists $K>0$ such that
		\begin{equation*} \label{Holder continuity of F}
			\sup_{(t,z)}\left\{\left| \mathbb{H}(t,\cdot,\cdot,z)\right|^{(\alpha)}_{[0,t\land\delta]\times\mathbb{R}^d}\right\}=K;
		\end{equation*}
		
		\item \textbf{(Lipschitz continuity)} for any $(t,s,y,z_1),(t,s,y,z_2)\in\Pi$, there exists $L>0$ such that
		\begin{equation*} \label{Lipschitz continuity of F} 
			|\mathbb{H}(t,s,y,z_1)-\mathbb{H}(t,s,y,z_2)|\leq L|z_1-z_2|. 
		\end{equation*}
	\end{enumerate} 
\end{assumption}

Under Assumption \ref{assumption}, we show the well-posedness of nonlocal quasilinear system \eqref{Nonlocal quasilinear systems}. 

\begin{theorem} \label{Well-posedness of quasilinear systems} 
	Suppose that Assumption \ref{assumption} holds, $\bm{g}\in\bm{\Omega}^{{(2r+\alpha)}}_{[0,T]}$, and  that \\ $\big(\partial_I\bm{g}_{|I|\leq 2r-1}(t,y),(\partial_I\bm{g})_{|I|\leq 2r-1}(s,y)\big)$ is ranged within a ball centered at $\overline{z}$ with radius $R_0/2$. Then there exists $\delta\in[0,T]$ such that the nonlocal quasilinear system \eqref{Nonlocal quasilinear systems} admits a unique solution $\bm{u}\in\bm{\Omega}^{(2r+\alpha)}_{[0,\delta]}$ in $\Delta[0,\delta]\times\mathbb{R}^d$.  
\end{theorem}
\begin{proof}
	We use fixed-point arguments to show the well-posedness of nonlocal quasilinear systems \eqref{Nonlocal quasilinear systems}. First, we construct a mapping $\bm{\Gamma}$ from $\bm{u}$ to $\bm{U}$, defined in a closed ball 
	\begin{equation*}
		\bm{\mathcal{U}}=\left\{\bm{u}\in\bm{\Omega}^{(2r+\alpha)}_{[0,\delta]}:\bm{u}(t,0,y)=\bm{g}(t,y),\lVert \bm{u}-\bm{g}\rVert^{(2r+\alpha)}_{[0,\delta]}\leq R\right\}, 
	\end{equation*}
	where $\bm{U}$ is the solution of the following nonlocal linear system: 
	\begin{equation} \label{Mapping of Nonlocal quasilinear systems}  
		\left\{
		\begin{array}{l}
			\bm{U}_s(t,s,y)=\sum\limits_{|I|= 2r}\bm{A}^{I}\Big(t,s,y,\left(\partial_J\bm{u}\right)_{|J|\leq 2r-1}(t,s,y), \\
			\qquad \qquad \qquad \qquad \qquad \qquad \quad \left(\partial_J\bm{u}\right)_{|J|\leq 2r-1}(s,s,y)\Big)\partial_I\bm{U}(t,s,y) \\
			\qquad\qquad\qquad
			+\sum\limits_{|I|= 2r}\bm{B}^{I}\big(t,s,y,\left(\partial_J\bm{u}\right)_{|J|\leq 2r-1}(t,s,y), \\
			\qquad \qquad \qquad \qquad \qquad \qquad \qquad \quad \left(\partial_J\bm{u}\right)_{|J|\leq 2r-1}(s,s,y)\big)\partial_I\bm{U}(s,s,y) \\
			\qquad\qquad\qquad   
			+\bm{F}\big(t,s,y,\left(\partial_J\bm{u}\right)_{|J|\leq 2r-1}(t,s,y),\left(\partial_J\bm{u}\right)_{|J|\leq 2r-1}(s,s,y)\big), \\
			\bm{U}(t,0,y)=\bm{g}(t,y),\hfill 0\leq s\leq t\leq T,\quad y\in\mathbb{R}^d. 
		\end{array}
		\right. 
	\end{equation}
	To validate the term $\mathbb{H}\big(t,s,y,\left(\partial_I\bm{u}\right)_{|I|\leq 2r-1}(t,s,y),  \left(\partial_I\bm{u}\right)_{|I|\leq 2r-1}(s,s,y)\big)$, it requires that the range of various derivatives of $\bm{u}$ in $\bm{\mathcal{U}}$ is contained in $B(\overline{z},R_0)$. Specifically, since
	\begin{equation*} \label{Range of u}
		\sup\limits_{\Delta[0,\delta]\times\mathbb{R}^d}\bigg\{\sum_{|I|\leq 2r}\left(\left|\partial_I\bm{u}(t,s,y)-\partial_I\bm{g}(t,y)\right|+\left|\partial_I\bm{u}(s,s,y)-\partial_I\bm{g}(s,y)\right|\right)\bigg\}\leq C\delta^\frac{\alpha}{2r}R, 
	\end{equation*}
	we require the choices of $\delta$ and $R$ allowing that $C\delta^\frac{\alpha}{2r}R\leq R_0/2$. With such choices, the mapping $\bm{\Gamma}(\bm{u})$ defined by \eqref{Mapping of Nonlocal quasilinear systems} is well-defined given the well-posedness of \eqref{Nonlocal linear systems}.
	
	Second, we shall show that $\bm{\Gamma}(\bm{u})$ is a $\frac{1}{2}$-contraction. For any $\bm{u}$, $\bm{\widehat{\bm{u}}}\in\bm{\mathcal{U}}$, we denote $\bm{U}=\bm{\Gamma}(\bm{u})$ and $\widehat{\bm{U}}=\bm{\Gamma}(\widehat{\bm{u}})$. Then let us consider the nonlocal system satisfied by the difference between $\bm{U}$ and $\widehat{\bm{U}}$, $\Delta\bm{U}:=\bm{U}-\widehat{\bm{U}}$,    
	\begin{equation} \label{Difference between U and U hat}  
		\left\{
		\begin{array}{l}
			(\Delta\bm{U})_s(t,s,y)=\sum\limits_{|I|= 2r}\bm{A}^{I}\big(\bm{u}\big)\partial_I\big(\Delta\bm{U}\big)(t,s,y)+\sum\limits_{|I|= 2r}\bm{B}^{I}\big(\bm{u}\big)\partial_I\big(\Delta\bm{U}\big)(s,s,y) \\
			\qquad \qquad \qquad  +\sum\limits_{|I|=2r}\partial_{I}\widehat{\bm{U}}(t,s,y)\bigg(\sum\limits_{|J|\leq 2r-1}\int^1_0\partial_J\bm{A}^I(\bm{\theta}_\sigma)\big(\partial_J\bm{u}-\partial_J\widehat{\bm{u}}\big)(t,s,y)d\sigma \\
			
			\qquad\qquad\qquad\qquad\qquad\qquad\qquad
			+\sum\limits_{|J|\leq 2r-1}\int_0^1\partial_J\overline{\bm{A}}^I(\bm{\theta}_\sigma)\big(\partial_J\bm{u}-\partial_J\widehat{\bm{u}}\big)(s,s,y)d\sigma\bigg) \\
			
			\qquad \qquad \qquad +\sum\limits_{|I|=2r}\partial_{I}\widehat{\bm{U}}(s,s,y)\bigg(\sum\limits_{|J|\leq 2r-1}\int_0^1\partial_J\bm{B}^I(\bm{\theta}_\sigma)\big(\partial_J\bm{u}-\partial_J\widehat{\bm{u}}\big)(t,s,y)d\sigma \\
			
			\qquad\qquad\qquad\qquad\qquad\qquad\qquad
			+\sum\limits_{|J|\leq 2r-1}\int_0^1\partial_J\overline{\bm{B}}^I(\bm{\theta}_\sigma)\big(\partial_J\bm{u}-\partial_J\widehat{\bm{u}}\big)(s,s,y)d\sigma\bigg) \\
			
			\qquad\qquad \qquad +\sum\limits_{|J|\leq 2r-1}\int_0^1\partial_J\bm{F}(\bm{\theta}_\sigma)\big(\partial_J\bm{u}-\partial_J\widehat{\bm{u}}\big)(t,s,y)d\sigma \\
			
			\qquad\qquad \qquad +\sum\limits_{|J|\leq 2r-1}\int_0^1\partial_J\overline{\bm{F}}(\bm{\theta}_\sigma)\big(\partial_J\bm{u}-\partial_J\widehat{\bm{u}}\big)(s,s,y)d\sigma, \\
			
			\big(\bm{U}-\widehat{\bm{U}}\big)(t,0,y)=\bm{0},\hfill 0\leq s\leq t\leq T,\quad y\in\mathbb{R}^d,
		\end{array}
		\right. 
	\end{equation} 
	where $\partial_J\bm{A}^I$, $\partial_J\bm{B}^I$, and $\partial_J\bm{F}$ (resp. $\partial_J\overline{\bm{A}}^I$, $\partial_J\overline{\bm{B}}^I$, and $\partial_J\overline{\bm{F}}$) are the first order derivatives of $\bm{A}^I$, $\bm{B}^I$, and $\bm{F}$, respectively, with respect to the variable $\partial_J\bm{u}(t,s,y)$ (resp. $\partial_J\bm{u}(s,s,y)$), while they are all evaluated at $(t,s,y,\bm{\theta}_\sigma)$ with 
	\begin{equation*}
		\begin{split}
			\bm{\theta}_\sigma(t,s,y)&=\sigma\big(\left(\partial_I\bm{u}\right)_{|I|\leq 2r-1}(t,s,y),  \left(\partial_I\bm{u}\right)_{|I|\leq 2r-1}(s,s,y)\big) \\
			&\qquad\qquad\qquad  
			+(1-\sigma)\big(\left(\partial_I\widehat{\bm{u}}\right)_{|I|\leq 2r-1}(t,s,y),  \left(\partial_I\widehat{\bm{u}}\right)_{|I|\leq 2r-1}(s,s,y)\big). 
		\end{split}
	\end{equation*}
	According to the Schauder estimate \eqref{Estimates of solutions of nonlocal linear PDEs} of nonlocal linear systems \eqref{Nonlocal linear systems} in Theorem \ref{Schauder estimates}, the definition of $\bm{\Gamma}(\cdot)$ \eqref{Mapping of Nonlocal quasilinear systems} tells us that
	\begin{equation*}
		\lVert\bm{U}\rVert^{(2r+\alpha)}_{[0,\delta]}=\lVert\bm{\Gamma}(\bm{u})\rVert^{(2r+\alpha)}_{[0,\delta]}\leq C(R), \quad  \lVert\widehat{\bm{U}}\rVert^{(2r+\alpha)}_{[0,\delta]}=\lVert\bm{\Gamma}(\widehat{\bm{u}})\rVert^{(2r+\alpha)}_{[0,\delta]}\leq C(R).
	\end{equation*}
	Again by applying the Schauder estimate \eqref{Estimates of solutions of nonlocal linear PDEs} in Theorem \ref{Schauder estimates} to \eqref{Difference between U and U hat}, we have 
	\begin{equation*} \label{Contraction}
		\lVert\bm{\Gamma}(\bm{u})-\bm{\Gamma}(\widehat{\bm{u}})\rVert^{(2r+\alpha)}_{[0,\delta]}=\lVert\bm{U}-\widehat{\bm{U}}\rVert^{(2r+\alpha)}_{[0,\delta]}\leq C(R)\delta^\frac{\alpha}{2r}\lVert\bm{u}-\widehat{\bm{u}}\rVert^{(2r+\alpha)}_{[0,\delta]}.
	\end{equation*} 
	Our desire to have a $\frac{1}{2}$-contraction for $\bm{\Gamma}$ requires that the choices of $\delta$ and $R$ should satisfy $C(R)\delta^\frac{\alpha}{2r}\leq 1/2$.
	
	Third, a suitably large $R$ needs to be chosen such that $\bm{\Gamma}=\bm{\Gamma}(\bm{u})$ maps the closed ball $\bm{\mathcal{U}}$ into itself. For any $\bm{u}\in\bm{\mathcal{U}}$, we have 
	\begin{equation*}
		\lVert\bm{\Gamma}(\bm{u})-\bm{g}\rVert^{(2r+\alpha)}_{[0,\delta]}\leq \lVert\bm{\Gamma}(\bm{u})-\bm{\Gamma}(\bm{g})\rVert^{(2r+\alpha)}_{[0,\delta]}+\lVert \bm{\Gamma}(\bm{g})-\bm{g}\rVert^{(2r+\alpha)}_{[0,\delta]}\leq \frac{R}{2}+\lVert \bm{\Gamma}(\bm{g})-\bm{g}\rVert^{(2r+\alpha)}_{[0,\delta]}. 
	\end{equation*} 
	Hence, to ensure $\lVert\bm{\Gamma}(\bm{u})-\bm{g}\rVert^{(2r+\alpha)}_{[0,\delta]}\leq R$, we require the choices of $\delta$ and $R$ allowing that $\lVert \bm{\Gamma}(\bm{g})-\bm{g}\rVert^{(2r+\alpha)}_{[0,\delta]}\leq R/2$.
	
	Finally, by striking the balance between large enough $R$ and small enough $\delta$ according to the aforementioned three conditions, $\bm{\Gamma}$ is a contraction mapping $\bm{\mathcal{U}}$ into itself. Therefore, the Banach fixed-point theorem asserts that it has a unique fixed point $\bm{u}$ in $\bm{\mathcal{U}}$ and clearly, $\bm{u}$ satisfying $\bm{u}=\bm{\Gamma}(\bm{u})$ solves \eqref{Nonlocal quasilinear systems}.
	
	Until now, we proved the existence of solutions of \eqref{Nonlocal quasilinear systems}. For the uniqueness of the solutions to \eqref{Nonlocal quasilinear systems}, it follows with similar arguments in the proof of \cite{Lei2021,Lei2021a} for nonlocal linear systems, or else directly from the Schauder estimate for the nonlocal, homogeneous, linear, strongly parabolic system with initial value zero, which drives the difference of any two solutions $\bm{u}$ and $\overline{\bm{u}}$ in $\bm{\Omega}^{(2r+\alpha)}_{[0,\delta]}$ to the quasilinear system \eqref{Nonlocal quasilinear systems}.
\end{proof}

Since the quasilinearization method in Subsection \ref{sec:degeneration} requires the solutions of \eqref{Nonlocal fully-nonlinear systems} and \eqref{Nonlocal quasilinear systems} to be $(2r+1)$-continuously differentiable with respect to $y$ such that the degeneracy transformation is available, we need to improve appropriately the regularities of $\bm{A}^I$, $\bm{B}^I$, $\bm{f}$ and $\bm{g}$ such that the requirement is fulfilled.  

\begin{corollary} \label{Higher regularities} 
	Suppose that Assumption \ref{assumption} holds while $\bm{g}\in\bm{\Omega}^{{(2r+1+\alpha)}}_{[0,T]}$ and $\mathbb{H}$ is required to be $(\frac{1+\alpha}{2r},1+\alpha)$-H\"{o}lder continuity with respect to $s$ and $y$ uniformly. Then, there exist $\delta>0$ and a unique $\bm{u}\in\bm{\Omega}^{{(2r+1+\alpha)}}_{[0,\delta]}$ satisfying \eqref{Nonlocal quasilinear systems} in $\Delta[0,\delta]\times\mathbb{R}^d$.  
\end{corollary}
Since the consideration of higher levels of regularity is nested in Theorem \ref{Schauder estimates}, the proof of Corollary \ref{Higher regularities} follows directly by the proof of Theorem \ref{Well-posedness of quasilinear systems} and is thus omitted.

Consequently, the solvability of nonlocal quasilinear systems \eqref{Nonlocal quasilinear systems} implies the well-posedness of a class of nonlocal fully-nonlinear systems \eqref{Nonlocal fully-nonlinear systems} with smooth enough nonlinearity $\bm{F}$ and initial condition $\bm{g}$. In addition, the well-posedness results of nonlocal semilinear systems \eqref{Nonlocal semilinear systems} can be obtained similarly and the proof is simpler due to the independence of $\bm{A}^I$ and $\bm{B}^I$ of $\bm{u}$. 


\subsection{A Variant of Quasilinearization for the Well-posedness Results} \label{sec:Variant} 
In this subsection, we illustrate a variant of the quasilinearization method, which provides more insights into the related studies. In Subsection \ref{sec:degeneration}, the quasilinearization is achieved by considering the differentiation with respect to the spatial variables. Here, we consider a quasilinearization with differentiating with respect to the temporal variables and in the nonlocal framework, they refer to $(t,s)$ in \eqref{Nonlocal fully-nonlinear systems}. We still set the goal to prove the well-posedness of \eqref{Nonlocal fully-nonlinear systems} or the degenerated nonlocal quasilinear systems.

For simplicity, we consider again the simplied version of nonlocal fully-nonlinear systems \eqref{Simplied nonlocal fully-nonlinear equation}. Suppose that $u$ is a solution of \eqref{Simplied nonlocal fully-nonlinear equation} with twice continuously differentiability with respect to $s$, then the corresponding $\left(\frac{\partial u}{\partial s},\frac{\partial u}{\partial t}\right)(t,s,y)$ satisfy the following nonlocal system: 
\begin{equation*}  
	\left\{
	\begin{array}{l}
		\left(\frac{\partial u}{\partial s}\right)_s(t,s,y)=F_s(u)+\sum^d\limits_{i,j=1}F_{q_{ij}}(u)\left(\frac{\partial u}{\partial s}\right)_{y_iy_j}(t,s,y) \\
		
		\qquad\qquad\qquad\quad +\sum^d\limits_{i,j=1}F_{n_{ij}}(u)\left(\frac{\partial u}{\partial s}\right)_{y_iy_j}(s,s,y)+\sum^d\limits_{i,j=1}F_{n_{ij}}(u)\left(\frac{\partial u}{\partial t}\right)_{y_iy_j}(s,s,y), \\
		
		\left(\frac{\partial u}{\partial t}\right)_s(t,s,y)=F_t(u)+\sum^d\limits_{i,j=1}F_{q_{ij}}(u)\left(\frac{\partial u}{\partial t}\right)_{y_iy_j}(t,s,y), \\
		
		\left(\frac{\partial u}{\partial s},\frac{\partial u}{\partial t}\right)(t,0,y)=\left(F\big(t,0,y,g_{yy}(t,y),g_{yy}(0,y)\big),g_t(t,y)\right),\\
		\hfill 0\leq s\leq t\leq T,\quad y\in\mathbb{R}^d. 
	\end{array}
	\right. 
\end{equation*} 
Inspired by the system above, we construct a mapping $u\mapsto\mathcal{M}(u):=\left(w^{(1)},w^{(2)}\right)$, where $\left(w^{(1)},w^{(2)}\right)$ solves the following coupled linear system: 
\begin{equation} \label{Mapping from u to us and ut} 
	\left\{
	\begin{array}{l}
		w^{(1)}_s(t,s,y)=F_s(u)+\sum^d\limits_{i,j=1}F_{q_{ij}}(u)w^{(1)}_{y_iy_j}(t,s,y) \\
		
		\qquad\qquad\qquad +\sum^d\limits_{i,j=1}F_{n_{ij}}(u)w^{(1)}_{y_iy_j}(s,s,y)+\sum^d\limits_{i,j=1}F_{n_{ij}}(u)w^{(2)}_{y_iy_j}(s,s,y), \\
		
		w^{(2)}_s(t,s,y)=F_t(u)+\sum^d\limits_{i,j=1}F_{q_{ij}}(u)w^{(2)}_{y_iy_j}(t,s,y), \\
		
		\left(w^{(1)},w^{(2)}\right)(t,0,y)=\left(F\big(t,0,y,g_{yy}(t,y),g_{yy}(0,y)\big),g_t(t,y)\right),\\
		\hfill 0\leq s\leq t\leq T,\quad y\in\mathbb{R}^d. 
	\end{array}
	\right. 
\end{equation} 
It is clear that given a regular enough $u$, the second equation of \eqref{Mapping from u to us and ut} is a conventional local parabolic PDE for $w^{(2)}$ parameterized by $t$. After acquiring the solution of $w^{(2)}$, the first equation of \eqref{Mapping from u to us and ut} is a standard nonlocal linear PDE for $w^{(1)}$. From another viewpoint, the system \eqref{Mapping from u to us and ut} can be considered as a nonlocal linear system for $\left(w^{(1)},w^{(2)}\right)$ of the form \eqref{Nonlocal linear systems}. Hence, the mapping $\mathcal{M}$ is well-defined given the well-posedness of nonlocal linear systems \eqref{Nonlocal linear systems}.

By applying the Schauder estimate \eqref{Estimates of solutions of nonlocal linear PDEs} to \eqref{Mapping from u to us and ut}, we have 
\begin{equation*}
	\max\left\{\lVert w^{(1)}\rVert^{(2+\alpha)}_{[0,\delta]},\lVert w^{(2)}\rVert^{(2+\alpha)}_{[0,\delta]}\right\}\leq\lVert\left(w^{(1)},w^{(2)}\right)\rVert^{(2+\alpha)}_{[0,\delta]}\leq C(R),  
\end{equation*}
provided that $u$ belongs to a closed ball $B(g,R)$ of $\Omega^{(2+\alpha)}_{[0,\delta]}$ center at $g$ with radius $R$, where $\Omega^{(2+\alpha)}_{[0,\delta]}$ is the one-dimensional version of $\bm{\Omega}^{(2+\alpha)}_{[0,\delta]}$, i.e. $m=1$. For any $u$, $\widehat{u}\in B(g,R)$, we consider  
\begin{equation*}  
	\left\{
	\begin{array}{lr}
		\left(w^{(1)}-\widehat{w}^{(1)}\right)_s(t,s,y)=F_s(u)-F_s(\widehat{u}) \\
		\qquad\qquad\qquad\qquad\qquad\qquad
		+\sum^d\limits_{i,j=1}F_{q_{ij}}(u)\left(w^{(1)}-\widehat{w}^{(1)}\right)_{y_iy_j}(t,s,y) \\
		\qquad\qquad\qquad\qquad\qquad\qquad +\sum^d\limits_{i,j=1}\left(F_{q_{ij}}(u)-F_{q_{ij}}(\widehat{u})\right)\widehat{w}^{(1)}_{y_iy_j}(t,s,y), \\
		\qquad\qquad\qquad\qquad\qquad\qquad +\sum^d\limits_{i,j=1}F_{n_{ij}}(u)\left(w^{(1)}-\widehat{w}^{(1)}\right)_{y_iy_j}(s,s,y) \\
		\qquad\qquad\qquad\qquad\qquad\qquad +\sum^d\limits_{i,j=1}\left(F_{n_{ij}}(u)-F_{n_{ij}}(\widehat{u})\right)\widehat{w}^{(1)}_{y_iy_j}(s,s,y), \\
		\qquad\qquad\qquad\qquad\qquad\qquad +\sum^d\limits_{i,j=1}F_{n_{ij}}(u)\left(w^{(2)}-\widehat{w}^{(2)}\right)_{y_iy_j}(s,s,y) \\
		\qquad\qquad\qquad\qquad\qquad\qquad +\sum^d\limits_{i,j=1}\left(F_{n_{ij}}(u)-F_{n_{ij}}(\widehat{u})\right)\widehat{w}^{(2)}_{y_iy_j}(s,s,y), \\
		
		\left(w^{(2)}-\widehat{w}^{(2)}\right)_s(t,s,y)=F_t(u)-F_t(\widehat{u}) \\
		\qquad\qquad\qquad\qquad\qquad\qquad +\sum^d\limits_{i,j=1}F_{q_{ij}}(u)\left(w^{(2)}-\widehat{w}^{(2)}\right)_{y_iy_j}(t,s,y) \\
		\qquad\qquad\qquad\qquad\qquad\qquad +\sum^d\limits_{i,j=1}\left(F_{q_{ij}}(u)-F_{q_{ij}}(\widehat{u})\right)\widehat{w}^{(2)}_{y_iy_j}(t,s,y), \\
		
		\left(w^{(1)}-\widehat{w}^{(1)},w^{(2)}-\widehat{w}^{(2)}\right)(t,0,y)=(0,0),\qquad 0\leq s\leq t\leq T,\quad y\in\mathbb{R}^d. 
	\end{array}
	\right. 
\end{equation*}
which is a nonlocal linear system of $\left(w^{(1)}-\widehat{w}^{(1)},w^{(2)}-\widehat{w}^{(2)}\right)$ given $u$, $\widehat{u}$ and $(\widehat{w}^{(1)},\widehat{w}^{(2)})$. Consequently, again by the Schauder estimate \eqref{Estimates of solutions of nonlocal linear PDEs}, we have 
\begin{eqnarray*}
	\max\left\{\lVert w^{(1)}-\widehat{w}^{(1)}\rVert^{(2+\alpha)}_{[0,\delta]},\lVert w^{(2)}-\widehat{w}^{(2)}\rVert^{(2+\alpha)}_{[0,\delta]}\right\} &\leq&\lVert\left(w^{(1)}-\widehat{w}^{(1)},w^{(2)}-\widehat{w}^{(2)}\right)\rVert^{(2+\alpha)}_{[0,\delta]} \\
	&\leq& C(R)\lVert u-\widehat{u}\rVert^{(2+\alpha)}_{[0,\delta]}.
\end{eqnarray*}

Next, we introduce another operator $\mathcal{N}((w^{(1)},w^{(2)}))=U$, where $U$ is given by 
\begin{equation*}
	U(t,s,y)=g(t,y)+\int^s_0 w^{(1)}(t,\tau,y)d\tau. 
\end{equation*}
It is clear that the operator $\mathcal{N}$ is well-defined and $U_t(t,s,y)=g_t(t,y)+\int^s_0 w^{(1)}_t(t,\tau,y)d\tau$. Hence, by choosing suitable $R$ and $\delta$, we have a priori estimate 
\begin{equation*}
	\begin{split}
		\lVert U-\widehat{U}\rVert^{(2+\alpha)}_{[0,\delta]}
		&\leq \left(\delta+\delta^{1-\frac{\alpha}{2}}\right)\lVert w^{(1)}-\widehat{w}^{(1)}\rVert^{(2+\alpha)}_{[0,\delta]}\leq 
		C(R)\left(\delta+\delta^{1-\frac{\alpha}{2}}\right)\left\lVert u-\widehat{u}\right\rVert^{(2+\alpha)}_{[0,\delta]} \\
		&\leq
		\frac{1}{2}\left\lVert u-\widehat{u}\right\rVert^{(2+\alpha)}_{[0,\delta]}.
	\end{split}
\end{equation*}
Therefore, the compound operator $U=\mathcal{N}\circ\mathcal{M}(u)$ can be constructed as a $\frac{1}{2}$-contraction mapping $B(g,R)$ to itself. It is easy to see that the fixed point $u$ satisfying $u=\mathcal{N}\circ\mathcal{M}(u)$ solves \eqref{Simplied nonlocal fully-nonlinear equation}. The uniqueness follows straightforwardly with arguments as in the proof of Theorem \ref{Well-posedness of quasilinear systems}. 

As a closing remark of this variant of quasilinearization method, we note that $u$ is required to be twice continuously differentiable with respect to $s$. In the classical PDE theory, it can be achieved easily by improving the regularities of coefficients, inhomogeneous term, and initial data. However, for nonlocal linear systems \eqref{Nonlocal linear systems}, there is a restriction in Theorem \ref{Schauder estimates}, $k\leq 2r-1$, leading to that the second-order derivatives of the solutions of \eqref{Nonlocal linear systems} with respect to $s$ are not guaranteed to be well-defined. This limitation is mainly caused by the fact that the study of nonlocal systems so far only requires a continuously differentiable property for the external temporal variable $t$ of prospective solutions $\bm{u}(t,s,y)$. We believe that the twice continuous differentiablity with respect to $s$ can be achieved by improving the regularity regarding $t$, which is put into our research agenda. 


\section{Conclusions} \label{sec:conclusion}
This paper studies nonlocal parabolic systems originated from time-inconsistent problems in the theory of stochastic differential equations and control/game problems. Instead of directly proving the well-posedness of nonlocal fully-nonlinear systems, we proposed quasilinearization methods that convert the nonlocal systems from fully-nonlinear to quasilinear ones. By leveraging on the results in \cite{Lei2021a}, we showed the equivalence of solvabilities between the nonlocal fully-nonlinear systems and the associated nonlocal quasilinear systems. We also adopt fixed-point arugments to show the well-posedness of a general class of the nonlocal quasilinear systems, which immediately implies that of the nonlocal fully-nonlinear systems.


%



\bibliographystyle{siamplain}
\bibliography{NonlocalityRef}
\end{document}